\documentclass{article}
\usepackage{graphicx} 
\usepackage{tikz} 
\usepackage{tikz-cd} 
\usepackage{amsthm} 
\usepackage{amsmath, amsfonts, amssymb,  mathtools} 
\usepackage{overpic}
\usepackage{subcaption}
\usepackage{url}

\usepackage{xcolor}
\usepackage[colorlinks=true,linkcolor=blue,citecolor=red,urlcolor=blue]{hyperref}
\usepackage[nameinlink,noabbrev]{cleveref}

\newcommand{\Addresses}{{
  \bigskip
  \footnotesize

  \textsc{Department of Mathematics, Georgia Institute of Technology,
    Atlanta, Georgia 30332}\par\nopagebreak
  \textit{E-mail address}: \texttt{seaneli@gatech.edu}

  \medskip

  \textsc{Department of Mathematics, Georgia Institute of Technology,
    Atlanta, Georgia 30332}\par\nopagebreak
  \textit{E-mail address}: \texttt{swan48@gatech.edu}
}}

\theoremstyle{plain}
\newtheorem{theorem}{Theorem}[section]
\newtheorem{conjecture}{Conjecture}[theorem]

\newtheorem{corollary}[theorem]{Corollary}

\newtheorem{proposition}[theorem]{Proposition}

\theoremstyle{definition}

\newtheorem{remark}[theorem]{Remark}

\DeclarePairedDelimiter\abs{\lvert}{\rvert}%
\DeclarePairedDelimiter\norm{\lVert}{\rVert}%
\DeclarePairedDelimiter\set{\{}{\}}%
\makeatletter
\let\oldabs\abs
\def\abs{\@ifstar{\oldabs}{\oldabs*}}
\let\oldnorm\norm
\def\norm{\@ifstar{\oldnorm}{\oldnorm*}}
\let\oldset\set
\def\set{\@ifstar{\oldset}{\oldset*}}
\makeatother

\newcommand{\cR}{{\mathcal{R}}}

\newcommand{\R}{\mathbb{R}}

\title{Exotic $\R^4$'s, RBG Links, and End Floer Homology 
}
\author{Sean Eli and Shunyu Wan}
\date{}
\begin{document}
\maketitle
\begin{abstract}
We give the first pair of non-diffeomorphic exotic $\R^4$'s made by attaching diffeomorphic Casson handles onto diffeomorphic disk complements. Our examples are obtained using the RBG link construction to find slice knots with diffeomorphic slice disk complements, but whose Whitehead doubled disk complements are not diffeomorphic. We distinguish the exotic $\R^4$'s using end Floer homology.
\end{abstract}

\section{Introduction}

In this paper we study \textit{exotic $\mathbb{R}^4$'s}: smooth manifolds homeomorphic to $\mathbb{R}^4$ but not diffeomorphic to $\mathbb{R}^4$. This bizarre, uniquely 4-dimensional phenomenon was first shown to occur in the 1980s, following the landmark theorems of Donaldson and Freedman \cite{donaldson, freedman}. Despite initial progress on the classification of exotic $\R^4$'s, many elementary questions remain open \cite{elihomlidman,  freedmantaylor, gompfmenagerie, gompfstipsicz, Taylor}.

We focus on the subfamily of \textit{slice $\R^4$'s}, made by attaching a Casson handle to the boundary of a slice disk complement and deleting the boundary \cite{bizacagompf, freedmandemichelis}. The way the smooth structure of a slice $\R^4$ depends on the {choices} of disk complement and Casson handle is one of the main open questions about exotic $\R^4$'s, and has seen considerable study \cite{bizacagompf, elihomlidman, gadgil, freedmandemichelis}. 
To date, all known examples that distinguish slice $\R^4$'s require varying either the slice disk complement or the Casson handle. In particular, it is unknown whether diffeomorphic disk complements and diffeomorphic Casson handles can be combined to produce non-diffeomorphic exotic $\R^4$'s. Our main result gives the first example of this phenomenon. Thus, the smooth structure of a slice $\mathbb{R}^4$ depends not only on the choice of disk complement and Casson handle, but also on the gluing map. 

\begin{theorem}\label{thm:exoticR4}
    There exist non-diffeomorphic exotic slice $\R^4$'s made by attaching diffeomorphic Casson handles onto diffeomorphic disk complements. 
\end{theorem}

We prove Theorem~\ref{thm:exoticR4} as follows. First, we use the unknotting number 1 RBG link technique introduced by Piccirillo \cite{piccirillo} (see also \cite{kegelspreer}) to construct slice knots $K_0=m(8_{20})$ and $J_0$ with diffeomorphic $0$-traces, but different maximal Maslov gradings in knot Floer homology (see Proposition~\ref{prop:knots}). By the trace embedding lemma, it follows that $K_0$ and $J_0$ admit diffeomorphic slice disk complements
in $B^4$ (see Corollary \ref{cor:disks}). We extend each disk complement to an exotic $\R^4$ by attaching the simplest Casson handle with all positive clasps, $CH_+$ \cite{gompfstipsicz}, to a meridian of the deleted knot, and removing the remaining boundary. The resulting slice $\R^4$'s are distinguished using graded end Floer homology \cite{elihomlidman, gadgil}. In the present setting, this invariant is controlled by the maximal nontrivial Maslov gradings of the knot Floer homologies of $K_0$ and $J_0$, due to recent work of Hom, Lidman, and the first author \cite{elihomlidman} (see Theorem~\ref{thm:endfloer}).

\begin{remark}
    In our example, the diffeomorphism $S^3_0(K_0)\cong S^3_0(J_0)$ coming from the RBG construction \textit{does not} send a meridian of $K_0$ to a meridian of $J_0$. Such a diffeomorphism would preserve the Casson handle attaching region, and then the diffeomorphism of the disk complements would extend over the Casson handle.
\end{remark} 

In \cite{elihomlidman}, it is shown that if $K_1$ and $K_2$ are nontrivial slice knots with different maximal Maslov gradings in knot Floer homology, then the maximal gradings of $HF^+(S^3_0(Wh(K_1)))$ and $HF^+( S^3_0(Wh(K_2)))$ differ, and this grading difference survives in the end Floer homology of the corresponding exotic $\R^4$'s. Hence, the proof of Theorem~\ref{thm:exoticR4} also gives the following phenomenon on the level of disk complements. 

\begin{theorem}\label{thm:doubleddisks}
There exist slice disks $(D_{K_0},K_0)$ and $(D_{J_0},J_0)$ in $(B^4,S^3)$ that have diffeomorphic complements, but whose Whitehead doubled disk complements are non-diffeomorphic, detected by the maximal nontrivial grading of $HF^+$ of their boundaries.
\end{theorem}
More generally, starting with an unknotting number 1 slice knot of genus $g \ge 2$ often yields a pair of exotic $\R^4$'s made from the same two pieces, as in Theorem~\ref{thm:exoticR4}. We confirm this for the knots $m(8_9), 10_{129}$, and $m(11n_{42})$. The genus requirement follows from work of Kegel and Spreer \cite{kegelspreer} who show that this RBG construction does not construct new knots from twisted Whitehead doubles.  

\begin{proposition}\label{prop:examples}
Let $K$ be any of the unknotting number 1 slice knots $m(8_{20}), m(8_9), 10_{129}$, or $m(11n_{42})$. Then $K$ shares a slice disk complement with some $0$-friend $J$, and the two exotic $\R^4$'s made by attaching $CH_+$ to a meridian of either deleted disk are not diffeomorphic.
\end{proposition}

\begin{remark}
    In Proposition~\ref{prop:examples}, various knots are mirrored in order to find a positive unknotting crossing. See Figure \ref{fig:unknotting}. In general this procedure can also be done for negative crossings. 
\end{remark}
Based on our observations, we make a conjecture about $\widehat{HFK}$ for certain $0$-friends. 
\begin{conjecture}
    Let $K$ be a homologically thin slice knot with genus $g\ge 2$, and unknotting number 1. Assume $K$ has a positive unknotting crossing and let $J$ be the corresponding Piccirillo friend. Then $\widehat{HFK}(J)$ is supported nontrivially in exactly $g$ diagonals, shifted up in Maslov grading from $K$. Hence the corresponding exotic $\R^4$ pairs are distinct.
\end{conjecture}

In Section~\ref{sec:tracedisk} we find $J_0$ by applying the unknotting number 1 RBG construction to $K_0 = m(8_{20})$, and show these knots admit diffeomorphic slice disk complements. In Section~\ref{sec:exoticR4} we prove Theorem~\ref{thm:exoticR4}, Theorem~\ref{thm:doubleddisks}, and Proposition~\ref{prop:examples}. Throughout, we only use positive, untwisted Whitehead doubles, and $CH_+$ denotes the simplest Casson handle with all positive clasps. 

\subsection{Acknowledgments}

We thank Marc Kegel for his help in finding the Piccirillo friend using SnapPy and SageMath. SE thanks John Etnyre for many helpful discussions and for his continuous support. We also thank B\"ulent Tosun, Marc Kegel, Kyle Hayden, Lisa Piccirillo, and Jen Hom for helpful discussions. SE was partially supported by NSF grant DMS-2203312. SW was partially supported by Georgia Tech postdoc funding.  

\section{Traces and disk complements}\label{sec:tracedisk}

  The \textit{$0$-trace} of a knot $K$ is the $4$-manifold constructed by attaching a $2$-handle to $B^4$ along $K$ with the $0$-framing. Given an unknotting number $1$ knot $K$, Piccirillo's RBG link technique produces another knot $J$ with the same $0$-trace as $K$. See \cite{kegelspreer, piccirillo} for details about this construction. The resulting knot $J$ is called a \textit{Piccirillo friend} of $K$, as in \cite{kegelspreer}. Note $K$ is slice if any only if any Piccirillo friend is slice, by the trace embedding lemma. 
    \begin{figure}[ht!]
\centering
       \includegraphics[width=\textwidth]{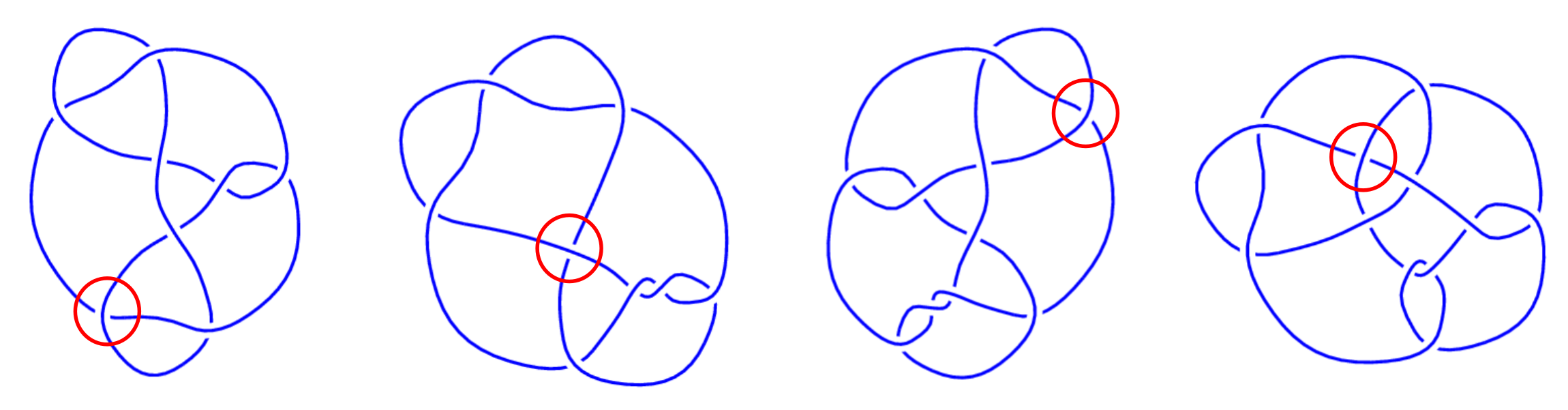}
\caption{Positive unknotting crossings for the slice knots $m(8_{20}), m(8_9), 10_{129}$, and $m(11n_{42})$ respectively. Figures from Knotinfo \cite{knotinfo}.}
\label{fig:unknotting}
\end{figure}

Let $K_0$ be the knot $m(8_{20})$, which has a positive unknotting crossing as indicated in Figure~\ref{fig:unknotting}. The left diagram of Figure~\ref{fig:RBG} is
the RBG link corresponding to the indicated crossing (namely, the lower left red-red crossing). By attaching a $1$-handle corresponding to the red circle, a $(-2)$-framed $2$-handle along the blue circle, and a $0$-framed $2$-handle along the green circle, we obtain a $4$-manifold which can be thought of as the $0$-trace of $K_0$, or as the $0$-trace of a Piccirillo friend, which we call $J_0$. Canceling the red-green pair yields the $0$-trace for $K_0$, but canceling the red-blue pair instead yields the $0$-trace for $J_0$. The knots $K_0$ and $J_0$ are shown in Figure~\ref{fig:RBG}.

     \begin{figure}[htb!]
\centering
        \begin{subfigure}[b]{0.32\textwidth}
            \begin{overpic}[width=\textwidth, 
 unit=1mm, tics=5]{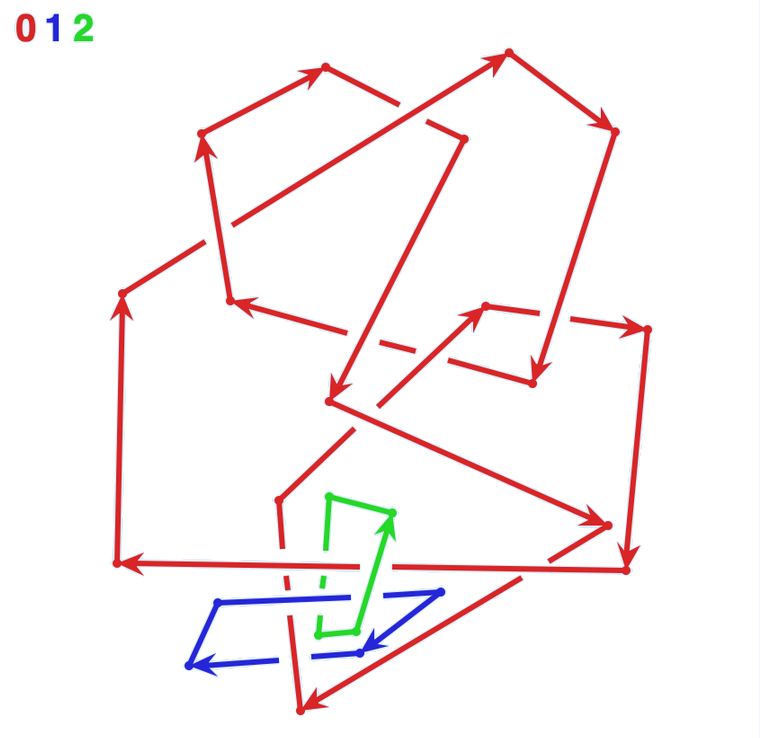}
            \end{overpic}
        \end{subfigure} \hspace{5mm} \begin{subfigure}[b]{0.27\textwidth}
            \begin{overpic}[width=\textwidth, 
 unit=1mm, tics=5]{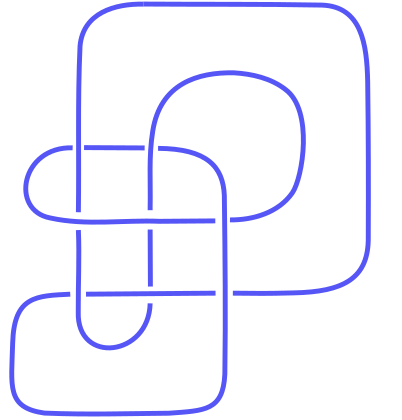}
            \end{overpic}
        \end{subfigure} \hspace{5mm}
        \begin{subfigure}[b]{0.27\textwidth}
            \begin{overpic}[width=\textwidth, 
 unit=1mm, tics=5]{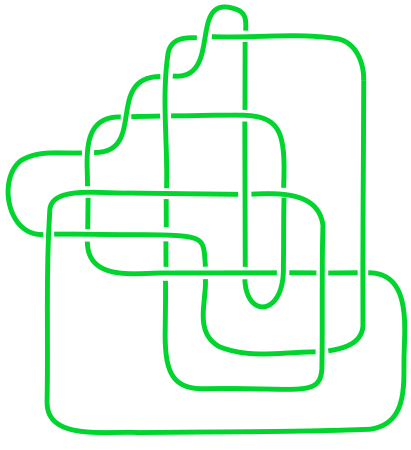}
 \end{overpic}
        \end{subfigure}
\caption{Left: RBG link obtained by applying the unknotting number 1 construction to a positive unknotting crossing of $m(8_{20})$, drawn as a SnapPy PLink input \cite{SnapPy}. Middle: the knot $K_0 = m(8_{20})$. Right: the knot $J_0$, which is a Piccirillo friend of $K_0$. Knot diagrams were obtained using SnapPy inside Sage \cite{sagemath}.}
\label{fig:RBG}
\end{figure}

Since $K_0$ and $J_0$ have diffeomorphic $0$-traces, they admit diffeomorphic slice disk complements. This fact is likely known to experts but we are not aware of a proof in the literature. We give a proof based on the \textit{smooth trace embedding lemma}, which states that a knot is smoothly slice if and only if its $0$-trace embeds smoothly in $S^4$. Indeed, the complement of the trace embedded in $S^4$ is the desired slice disk complement. 
\begin{theorem}\label{thm:traceembedding}
    Let $K$ be a knot in $S^3$ and $X_0(K)$ its $0$-trace. Suppose $i : X_0(K)\hookrightarrow S^4$ is a smooth embedding. Then there exists a smooth slice disk $(D,K)\hookrightarrow (B^4,S^3)$ such that $S^4\setminus \text{int}(i(X_0(K)))$ is {orientation-reversing} 
    diffeomorphic to the slice disk complement $B^4\setminus \nu(D)$.
\end{theorem}
\begin{proof}
    The $0$-trace $X_0(K)$ consists of a $0$-handle $h^0$ union a 2-handle $h^2$ attached along $K\subset \partial h^0$ with the $0$-framing.   
     Since $i : X_0(K)\hookrightarrow S^4$ is a smooth embedding, its restriction to the $0$-handle
    $i|_{h^0} : h^0\hookrightarrow S^4$ is a smooth embedding of a standard compact ball in $S^4$. Thinking of $S^4$ as the double $B^4\cup -B^4$, after smooth isotopy we may assume $i(h^0) = B^4$. Then the complement of $\text{int}(i(h^0))$ in $S^4$ is $-B^4$, and $i(h^2)$ is a tubular neighborhood of $i(C)$ in $-B^4$, where $C\subset h^2$ is the core 2-disk. Since $i(K)$ is the knot $m(K) \subset -B^4$, we see an embedding of a slice disk
    $(i(C), m(K)) \hookrightarrow (-B^4,-S^3)$. Thus $(-B^4)\setminus \nu(i(C))$ is a slice disk complement for $m(K)$, but this is the same as $S^4\setminus \text{int}(i(X_0(K)))$.   
\end{proof}

\begin{corollary}\label{cor:disks}
    Suppose $K,J$ are slice knots in $S^3$ with $X_0(K)\cong X_0(J)$. Then there exist slice disks $(D_K,K)\hookrightarrow (B^4,S^3)$, $(D_J,J)\hookrightarrow (B^4,S^3)$ such that $B^4\setminus \nu(D_K)\cong B^4\setminus \nu(D_J)$.
\end{corollary}

\begin{proof}
    Since $K$ is slice, the trace embedding lemma gives a smooth embedding $i: X_0(K)\hookrightarrow S^4$. Letting $\psi: X_0(J) \to X_0(K)$ be a diffeomorphism, we find that $j := i\circ \psi : X_0(J)\hookrightarrow S^4$ is a smooth embedding with the same image as $i$. Applying Theorem~\ref{thm:traceembedding} to $i$ shows there is a smooth slice disk $(D_K,K)\hookrightarrow (B^4,S^3)$ such that $\overline{S^4\setminus \text{int}(i(X_0(K)))}$ is diffeomorphic to $B^4\setminus \nu(D_K)$. Applying Theorem~\ref{thm:traceembedding} to $j$ instead gives a smooth slice disk $(D_J,J)\hookrightarrow (B^4,S^3)$ such that $\overline{S^4\setminus \text{int}(j(X_0(J)))}$ is diffeomorphic to $B^4\setminus \nu(D_J)$. But then
    \[B^4\setminus \nu(D_J)\cong \overline{S^4\setminus \text{int}(j(X_0(J)))} = \overline{S^4\setminus \text{int}(i(X_0(K)))} \cong B^4\setminus \nu(D_K). \]\end{proof}
    
\section{Distinguishing exotica with end Floer homology}\label{sec:exoticR4}
To distinguish slice $\R^4$'s, we use the following theorem of Hom, Lidman, and the first author \cite{elihomlidman}. In the following, $\widehat{HFK}_{red}$ is the submodule of $\widehat{HFK}$ obtained by modding out a generator that survives in the spectral sequence to $\widehat{HF}(S^3)$.

\begin{theorem}[Theorem 1.1 of \cite{elihomlidman}]\label{thm:endfloer}
Let $\cR$ be a slice $\R^4$ built by attaching $CH_+$ to a slice disk complement $(B^4,S^3)\setminus \nu(D^2,K)$ for a nontrivial slice knot $K$, along a $0$-framed meridian of the deleted disk, and removing the boundary. Then $\cR$ is not diffeomorphic to $\R^4$. Moreover, if $K_1$ and $K_2$ are two nontrivial slice knots whose knot Floer homology $\widehat{HFK}_{red}$ has different maximal nontrivial Maslov gradings, then the exotic $\R^4$'s $\cR_1$ and $\cR_2$ built as above are not diffeomorphic with any choice of orientations. 
\end{theorem}

In \cite{elihomlidman}, Theorem~\ref{thm:endfloer} was used to distinguish exotic $\R^4$'s made by attaching $CH_+$ to different slice disk complements. Theorem~\ref{thm:endfloer} is proven by computing the \textit{graded end Floer homology} of the end of $\cR$; the maximal Maslov grading of $K$ controls the maximal grading of $HF^+(S^3_0(Wh(K))$ and hence that of the graded end Floer homology.

\begin{proposition}\label{prop:knots}
    Let $K_0$ be the knot $m(8_{20})$, and $J_0$ its Piccirillo friend from Figure~\ref{fig:RBG}.
    Then $\widehat{HFK}_{red}(K_0)$ has maximal nontrivial Maslov grading 2 and $\widehat{HFK}_{red}(J_0)$ has maximal nontrivial Maslov grading 4.
\end{proposition}

\begin{proof}
    Working with SnapPy inside Sage, we used Szab\'o's knot Floer calculator to compute $\widehat{HFK}$ of both knots, which are plotted in Figure~\ref{fig:HFK}. Inputting the left diagram of Figure~\ref{fig:RBG} into SnapPy's PLink editor, and running the following commands, computes $\widehat{HFK}$ of $K_0$. Note, the first line opens the PLink editor, which is where the RBG diagram should be entered. 
    \begin{verbatim}
         L = snappy.Manifold()
         L.dehn_fill([(0,1),(0,0),(0,1)])
         K = L.filled_triangulation()
         P = K.exterior_to_link()
         P.knot_floer_homology()\end{verbatim}Swapping the second line with 
    \begin{verbatim}
        L.dehn_fill([(0,1),(-2,1),(0,0)])\end{verbatim} and keeping the other lines the same gives $\widehat{HFK}$ of $J_0$. Note, SnapPy’s
    exterior\_to\_link()
command recovers a knot diagram from the given exterior, based on the algorithm of Dunfield-Obeidin-Rudd \cite{DOR}. Since the computation is performed with an oriented triangulation and specified peripheral curves, the resulting diagram is determined up to this oriented exterior data; in particular, the procedure does not introduce an ambiguity between a knot and its mirror. Since $K_0$ and $J_0$ are slice, we see that for both knots the maximal nontrivial Maslov grading of $\widehat{HFK}_{red}$ is the same as that of $\widehat{HFK}$, and the Proposition follows. 
\end{proof}

 \begin{figure}[ht]
\centering
         \begin{overpic}[width=0.7\textwidth, 
 unit=1mm, tics=5]{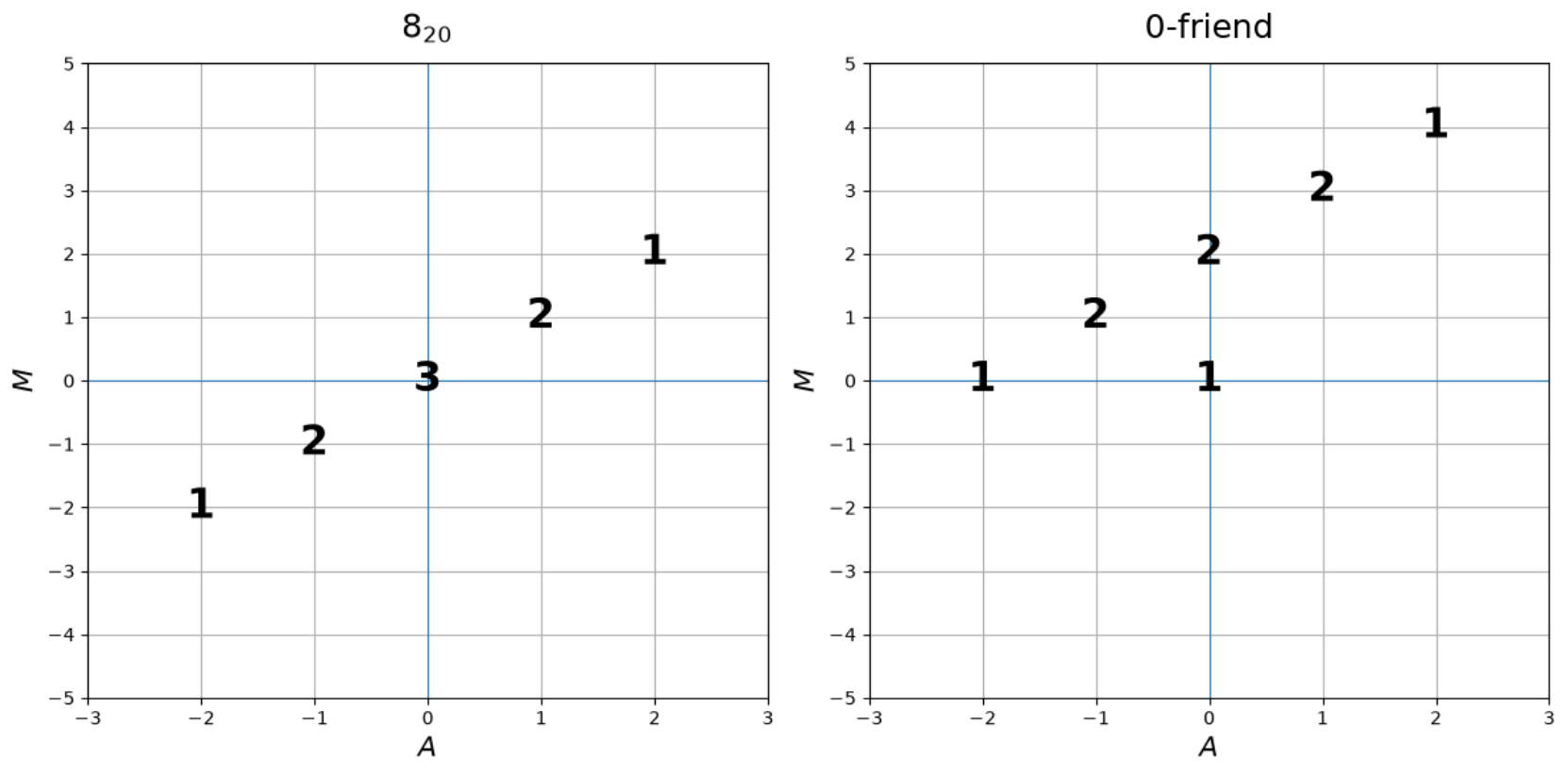}
 \end{overpic}
\caption{Left: Ranks of $\widehat{HFK}(K_0 = m(8_{20}))$. Right: Ranks of $\widehat{HFK}(J_0)$. The horizontal axis represents the Alexander grading and the vertical axis represents the Maslov grading.}
\label{fig:HFK}
\end{figure}

\begin{proof}[Proof of Theorem \ref{thm:exoticR4}]
    The knots $K_0$ and $J_0$ of Proposition \ref{prop:knots} admit slice disks $D_{K_0}$ and $D_{J_0}$ with diffeomorphic complements, by Corollary \ref{cor:disks}. Proposition \ref{prop:knots} shows the knot Floer homologies of $K_0$ and $J_0$ have different nontrivial maximal Maslov gradings, thus by Theorem \ref{thm:endfloer} the slice $\R^4$'s obtained by attaching Casson handles to their corresponding disk complements (i.e. writing the boundary as $S^3_0(K_0)$, and attaching the Casson handle along a meridian to $K_0$, versus writing the boundary as $S^3_0(J_0)$, and attaching the Casson handle along a meridian to $J_0$) are not diffeomorphic.
\end{proof}

\begin{proof}[Proof of Theorem~\ref{thm:doubleddisks}]
Let the slice disks $D_{K_0}$ and $D_{J_0}$ be as in the proof of Theorem~\ref{thm:exoticR4}. Since the maximal Maslov gradings of $\widehat{HFK}_{red}(K_0)$ and $\widehat{HFK}_{red}(J_0)$ differ, the proof of \cite[Theorem 1.1]{elihomlidman} shows the maximal gradings of $HF^+(S^3_0(Wh(K_0))$ and $HF^+(S^3_0(Wh(J_0))$ differ. Thus, the Whitehead doubled disk complements are distinct, detected by $HF^+$ of their boundaries.    
\end{proof}

\begin{remark}
     While one can also distinguish $0$-surgeries on Whitehead doubles using their $JSJ$ decompositions, we need the Heegaard Floer result in order to distinguish the exotic $\R^4$'s.
\end{remark}

\begin{proof}[Proof of Proposition~\ref{prop:examples}]
    The case of $m(8_{20})$ is covered in the proof of Theorem~\ref{thm:exoticR4}. Performing the RBG link technique on the indicated positive unknotting crossings in Figure~\ref{fig:unknotting} yields $0$-friends for each of the remaining three knots. The knot Floer homologies of all three pairs were computed as in Proposition~\ref{prop:knots} and their ranks are plotted in Figure~\ref{fig:plots}. The proof of Theorem~\ref{thm:exoticR4} now yields the claim.
\end{proof}

\begin{figure}[ht]
    \centering
    \begin{subfigure}[b]{0.6\textwidth}
            \begin{overpic}[width=\textwidth, 
 unit=1mm, tics=5]{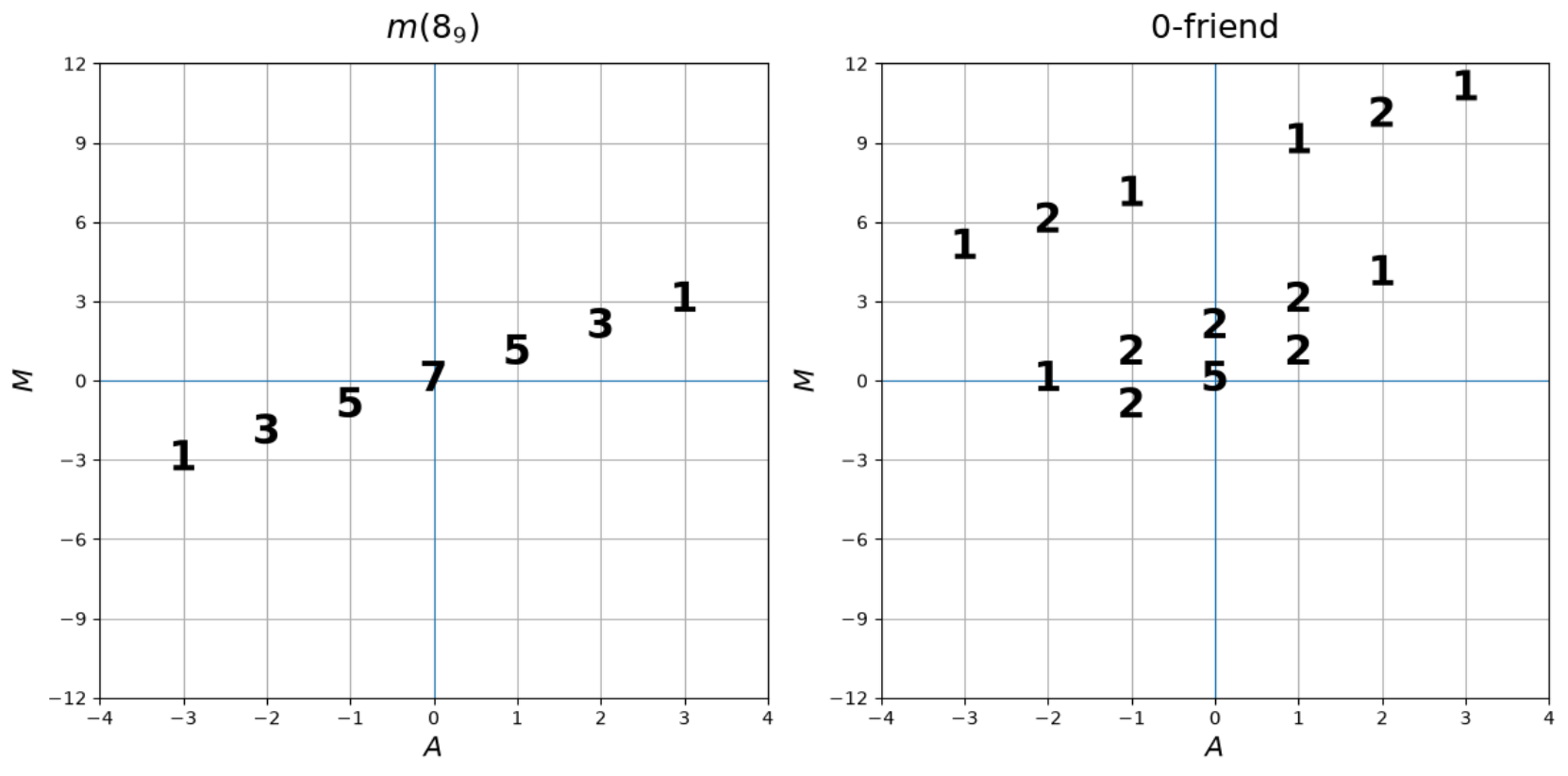}
 \end{overpic}
        \end{subfigure}
        \begin{subfigure}[b]{0.6\textwidth}
            \begin{overpic}[width=\textwidth, 
 unit=1mm, tics=5]{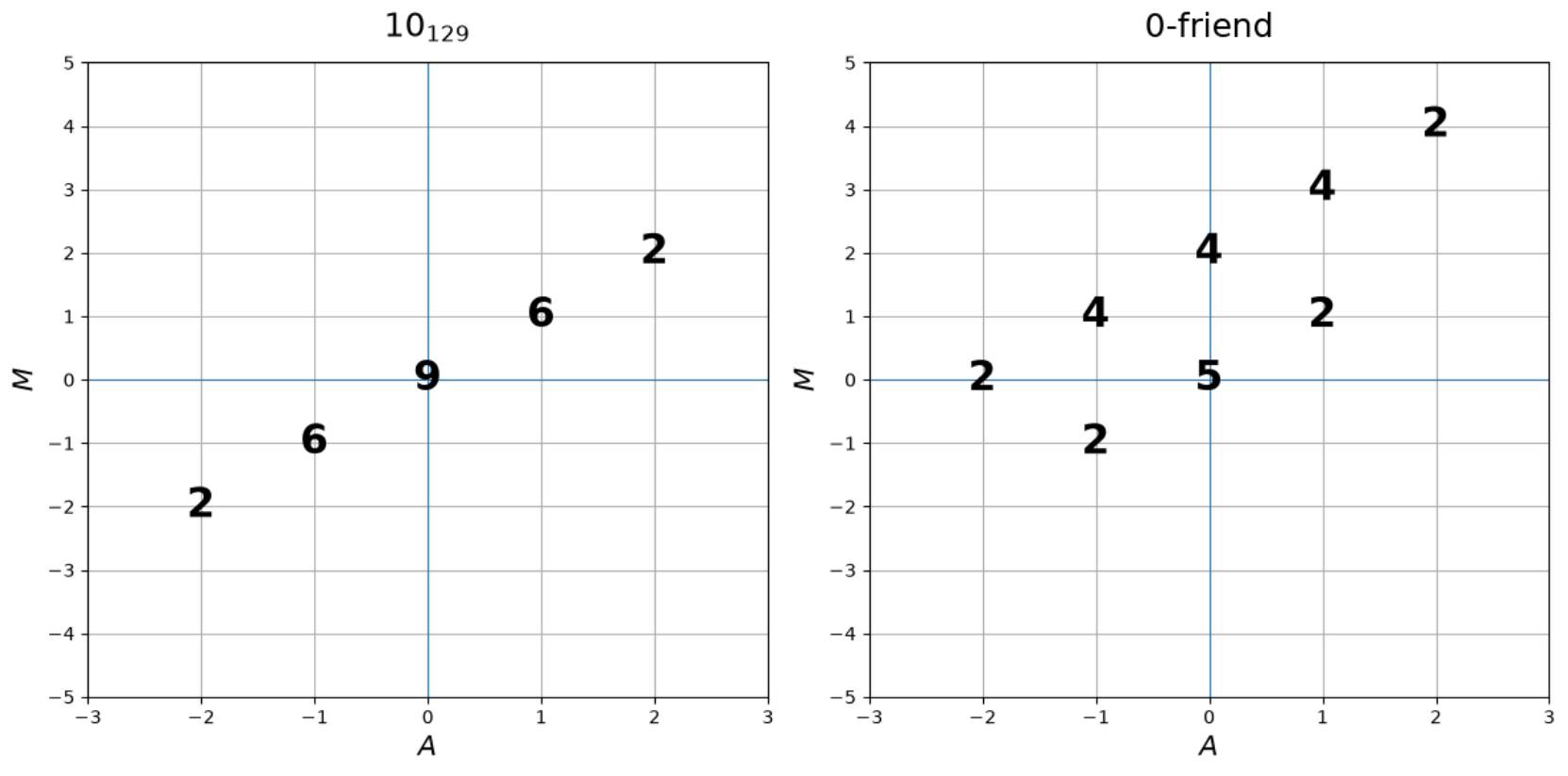}
 \end{overpic}
        \end{subfigure}
            \begin{subfigure}[b]{0.6\textwidth}
            \begin{overpic}[width=\textwidth, 
 unit=1mm, tics=5]{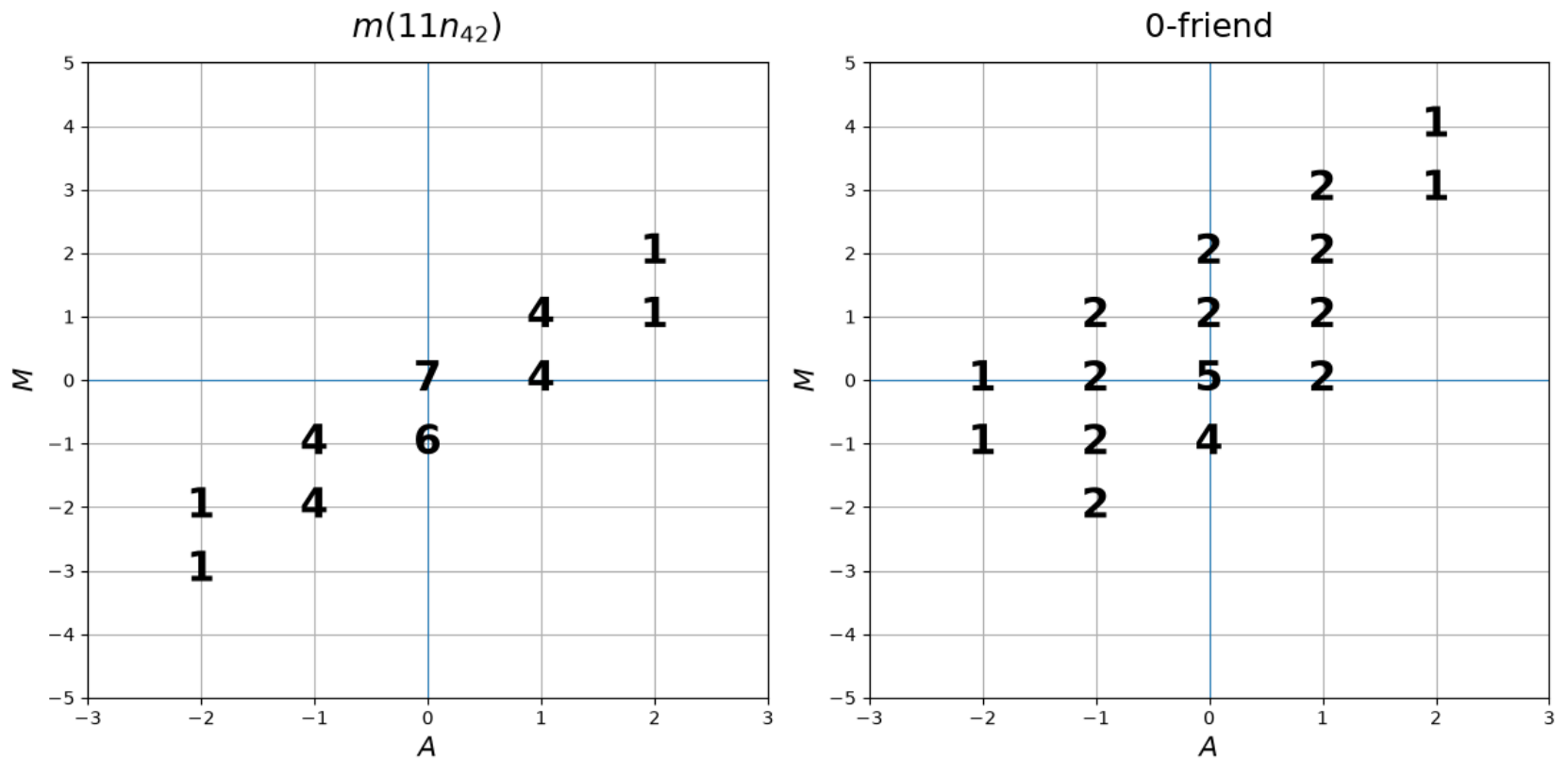} 
 \end{overpic}
        \end{subfigure}
\caption{Ranks of $\widehat{HFK}$ for $m(8_9), 10_{129}$, and $m(11n_{42})$ on the left, and of their corresponding Piccirillo friends (obtained using the unknotting crossings indicated in Figure~\ref{fig:unknotting}) on the right.}
\label{fig:plots}
\end{figure}

\bibliographystyle{abbrv}
\bibliography{mybib}
\Addresses
\end{document}